\documentclass[11pt]{amsart}
\usepackage{amsmath,amssymb,amsthm,amscd}\usepackage[usenames]{color}
\usepackage[all]{xy}
\usepackage{enumerate}
\usepackage{cancel}
\xyoption{all}
\usepackage{xcolor}
\usepackage{tikz}
\usepackage{tikz-cd}
\usepackage{graphicx}
\usepackage{pgf,tikz}
\usetikzlibrary{arrows}

\newtheorem{thm}{Theorem}[section]
\newtheorem{lem}[thm]{Lemma}
\newtheorem{cor}[thm]{Corollary}
\newtheorem{prop}[thm]{Proposition}

\theoremstyle{definition}
\newtheorem{rem}[thm]{Remark}
\newtheorem{defn}[thm]{Definition}

\newcommand{\QQ}{{\mathbb Q}}

\newcommand{\cA}{{\mathcal A}}

\newcommand{\cP}{{\mathcal P}}
\newcommand{\cR}{{\mathcal R}}

\newcommand{\lra}{\longrightarrow}
\newcommand{\ra}{\rightarrow}

\newcommand{\PP}{{\mathbb P}}

\DeclareMathOperator{\Ker}{Ker}

\DeclareMathOperator{\Pic}{Pic}
\DeclareMathOperator{\NS}{NS}

\DeclareMathOperator{\Image}{Im}
\DeclareMathOperator{\Id}{Id}
\DeclareMathOperator{\Aut}{Aut}
\DeclareMathOperator{\End}{End}

\title{Cyclic coverings of genus $2$ curves of Sophie Germain type}

\author{J.C. Naranjo$^{1,2}$}
 \address{Juan Carlos Naranjo \newline 1. Departament de Matem\`atiques i Inform\`atica,
Universitat de Barcelona, Gran Via de les Corts Catalanes, 585, 08007 Barcelona, Spain \newline 2. Centre de Recerca Matemàtica, Edifici C, Campus Bellaterra, 08193 Bellaterra, Spain }
 \email{jcnaranjo@ub.edu}

\author{A. Ortega}
\address{Angela Ortega \\ Institut f\"ur Mathematik, Humboldt Universit\"at zu Berlin \\ Germany}
\email{ortega@math.hu-berlin.de}

\author{I. Spelta}
\address{Centre de Recerca Matemàtica, Edifici C, Campus Bellaterra, 08193 Bellaterra, Spain}
\email{ispelta@crm.cat}

\thanks{The first and the third authors were partially supported by the Proyecto de Investigaci\'on PID2019-104047GB-100}
\date{}
\begin{document}

 \maketitle
 \section{Introduction}

We consider cyclic unramified coverings of degree $d$ of irreducible complex smooth genus-$2$ curves and their corresponding Prym varieties. They provide natural examples of polarized abelian varieties with automorphisms of order $d$.  The rich geometry of the associated Prym map, has been studied in several papers, see \cite{ries}, \cite{ortega}, \cite{lo11}, \cite {lo16}, \cite{albano pirola}, \cite{ag}, among others. Notice that the classical case, $d=2$, is completely explained in \cite[section 7]{mu} and \cite{naranjo}.  Nevertheless,
very few is known for higher values of $d$. In this article we investigate if the covering can be reconstructed from its Prym variety, that is, if the generic Prym Torelli Theorem holds for these coverings.

It is known that the Prym variety of an unramified cyclic covering (over any smooth curve) is isomorphic, as unpolarized varieties, to the product of two Jacobians (see \cite{ortega}) and when the 
degree $d$ is odd the Jacobians are isomorphic. 
Moreover, when $d$ is not prime, the existence of intermediate coverings gives a much more complicated scenario, which strongly depends on the decomposition of $d$ in primes. For this reason, as in \cite{ries}, we assume $d$ to be an odd prime.

Putting this in a modular setting, we consider the moduli space\footnote{Traditionally, $\mathcal{R}_g$ denotes the moduli space of \'etale double coverings over a genus $g$ curve. Since here we only consider genus 2 curves, we borrowed this notation for the degree-$d$ cyclic unramified coverings.} $\mathcal R_d$ of isomorphism classes of  cyclic unramified coverings $f:\widetilde C\rightarrow C$, with $g(C)=2$ and $\deg(f)=d$.  Equivalently,  $\mathcal R_d$ parametrizes isomorphism classes of pairs $(C,\langle \eta \rangle )$, where $\eta $ is a $d$-torsion point in $JC$ generating a subgroup $\langle \eta \rangle \cong \mathbb Z/d\mathbb Z$. We recall that the curve $\widetilde C$ is constructed by applying the  $\mathrm{Spec}$ functor to the sheaf of $\mathcal O_{C}$-algebras: 
\[
\mathcal O_C \oplus \eta \oplus \eta^2 \oplus \ldots \eta^{d-1}.
\]
Notice that $\widetilde C$ comes equipped with an automorphism $\sigma $ of order $d$ such that $C=\widetilde C/\langle \sigma \rangle $.

An important consequence of this construction is that it shows that any automorphism of $C$ leaving invariant $\langle \eta \rangle$ lifts to an automorphism on $\widetilde C$. This is the case of the hyperelliptic involution on $C$, which lifts to an involution $j$ on $\widetilde C$. Therefore, the dihedral group generated by $\sigma $ and $j$ acts on $\widetilde C$ providing an interesting geometric structure  in the several Jacobians and Pryms appearing naturally in the picture. We focus on the Prym variety $P(\widetilde C,C)$ defined as the component of the origin of the kernel of the norm map $J\widetilde C\rightarrow JC.$ It is a consequence of Riemann-Hurwitz Theorem that $g(\widetilde C)=d+1$, and thus $\dim P(\widetilde C,C)=d-1$. Moreover, the principal polarization on $J\widetilde C$ induces on $P(\widetilde C,C)$ a polarization $\tau $ of type $(1,\ldots ,1,d)$. We can define the Prym map as the map of moduli stacks:  
\[
\cP_d:\cR_{d} \lra \cA_{d-1}^{(1,\ldots,1,d)},
\]
which maps $f:\widetilde C\ra C$ to the isomorphism class of $(P(\widetilde C,C), \tau)$.

It is known that the generic fiber of $\cP_d$ is positive dimensional for $d=3,5$ (\cite{albano pirola}), for $d=7$ the degree onto its image is $10$ (\cite{lo16}) and the map is generically finite for $d\ge 7$ (\cite{ag}). In this paper, we prove:

\begin{thm} \label{main-thm} The Prym map $\mathcal P_d$ is generically injective for every prime $d\ge 11$ such that $k:=\frac {d-1}2$ is also prime.
\end{thm}

\begin{rem} It is conjectured that there are infinitely many pairs of prime numbers of the form $(k,2k+1)$. These are called Sophie Germain prime numbers. Under this hypothesis on $d$ and $k$, we will say that $f: \widetilde C \ra C$ is of Sophie Germain type.
\end{rem}

Our proof has ingredients of different nature. We use arithmetic arguments on abelian varieties of $GL_2$-type to analyse the endomorphism algebra of some Jacobians, combined with the use of theta-duality techniques inspired in  the Fourier-Mukai transform.

More precisely: the automorphism $\sigma $ on $\widetilde C$ induces and automorphism, denoted by the same letter,  
on $P:=P(\widetilde C,C)$ preserving the polarization $\tau $ and fixing point-wise the kernel of $\lambda_{\tau }:P\ra P^{\vee}$.
We prove first that  $\sigma$ is, generically, completely determined by $(P,\tau)$. Then, we consider the curve $C_0:=\widetilde C/\langle j \rangle$, where $j$ is a lifting of the hyperelliptic involution on $C$. In the second step, using a result of arithmetic nature (\cite{wu}), we  prove that for a generic covering, the only automorphisms of $C_0$ are the identity and, possibly, the hyperelliptic involution. 

Next, we consider the isomorphisms $JC_0\times JC_0 \ra P$  studied in \cite{ries} and \cite{ortega}. We prove, using step 2, that these isomorphisms are unique in general, which allow us  to recover canonically from $(P,\tau)$ the curve $C_0$ and a set of automorphisms $\beta_i$ on $JC_0$.

Finally, we  show how to reconstruct the covering $f:\widetilde C\ra C$ from these data. The key argument is that the whole diagram (\ref{global_diagram}) is determined by the map $h_0:C_0\ra \mathbb P^1$. We recover explicitly the fibers of $h_0$ in the following way: we fix a point $x\in C_0$ to embed $C_0$ in $JC_0$ and then we compute the theta dual (as introduced in \cite{pp} and used in \cite{ln}) of the curves $\beta_i(C_0)\subset JC_0$. This gives a translation of the Brill-Noether locus $W_{k-3}(C_0)$ by an effective degree $2$ divisor defined by two points in the fiber $h_0^{-1}(h_0(x))$. Varying $i$, we recover the whole fiber of $h_0$ at $h_0(x)$, and this ends the proof.

In the last sections of the paper, we consider the cases $d=9$ and $d=13$, which illustrate that without our assumptions the generic injectivity requires a case-by-case analysis which depends on the decomposition of $d$ and $k$ in prime numbers. 

In the case $d=9$, since $d$ is not prime, we have additional curves in the diagram. In this case, a deeper analysis of the automorphisms that appear in step 3 combined with  Galois theory arguments allow us to conclude the following (see Theorem \ref{thm_d=9}):

\begin{thm} The Prym map $\mathcal P_{9}$ is generically injective.
\end{thm}

In the case $d=13$, we are able to reduce the proof to the following question:

\vskip 3mm
\noindent
\textbf{Question:} Does there exist a $3$-dimensional family of double coverings of irreducible smooth curves $\widetilde {\mathcal C}/B\ra \mathcal {C} /B$ such that: $g(\widetilde {\mathcal {C}_b})=6$, $g( \mathcal {C}_b)=3$ and $P(\widetilde {\mathcal {C}_b},\mathcal {C}_b)$ is isogenous to $J\mathcal {C}_b$ for all $b\in B$?

\vskip 3mm
We think that this question has a negative answer. In that case, our results would imply the generic injectivity of $\mathcal P_{13}$.  

\vskip 5mm
\textbf{Acknowledgements:} We thank X. Guitart for valuable information about abelian varieties of $GL_2$-type and B. van Geemen, G. Farkas and G.P. Pirola for stimulating discussions on the subject. This material is based upon work supported by the Swedish Research Council under grant no. 2021-06594 while the second author was in residence at Institut Mittag-Leffler in Djursholm, Sweden during the semestre September-December of 2021.

\section{Set-up and notations}

The content of this section is borrowed from \cite{ries} and \cite{ortega}. We will state the results of these two papers without further quoting.
Let  $f: \widetilde C \ra C$ be  a cyclic $d$-covering of a curve $C$ of genus $2$ associated to a non-trivial $d$-torsion point $\eta \in JC[d]$. We denote by  $\sigma $ both the automorphism of order $d$ on $\widetilde C$ and the induced automorphism on $J\widetilde C$. The Prym variety of the covering, $P:=P(\widetilde C,C)$, is the component of the origin of the kernel of the norm map. One easily checks that $\sigma $ leaves $P$ invariant, so we keep the notation $\sigma $ for the  restriction to $P$. The polarization on $J\widetilde C$ induces on $P$ a polarization $\tau$ of  type $(1,\ldots ,1,d)$ which is invariant by $\sigma $, that is, there is a line bundle $L\in \Pic(P)$ representing $\tau \in \NS(P)$ such that $\sigma ^*(L)\cong L$.  Moreover, we have:

\begin{lem}\label{fixed_points}
The set of fixed points of $\sigma $ is exactly the Kernel $K(\tau)$, of the polarization map $\lambda_{\tau}:P\ra P^{\vee}$. Moreover, $K(\tau)= P\cap f^*(JC)$.
\end{lem}

From now on, we assume that $d$ is an odd prime, and we set $d=2k+1$.

The hyperelliptic involution $\iota $ on $C$ lifts  to an involution $j$ on $\widetilde C$. The dihedral group with $2d$ elements generated by $j$ and
$\sigma$ acts on $\widetilde C$. Notice that all the automorphisms $j\circ \sigma ^i$ are  involutions on $\widetilde C$ lifting $\iota$. We define 
the following curves:
\[
C_0:=\widetilde C/\langle j \rangle, \quad
C_1:=\widetilde C/\langle j\sigma \rangle, \quad
\ldots \quad
C_{d-1}:=\widetilde C/\langle j\sigma^{d-1} \rangle.
\]
And we denote by $\pi_i:\widetilde C\ra C_i$, $i=0,\ldots , d-1$ the quotient maps. These curves fit in the following commutative diagram:
\begin{equation}\label{global_diagram}
 \begin{tikzcd}
 	&&\widetilde C \arrow[dll,swap,"f"] \arrow[dl,"\pi_0"] 
  \arrow[d,"\pi_1"]
  \arrow[drr,"\pi_{d-1}"]
  &&
  \\
 	C   \arrow [drr,swap,"\varepsilon"]& C_0 
  \arrow[dr,"h_0"]
  & C_1 \arrow [d,"h_1"] & \ldots & C_{d-1} \arrow[dll,"h_{d-1}"]  
    \\
  &&\mathbb P^1 &&
 \end{tikzcd}
 \end{equation}
where the maps $\varepsilon, \pi_0,\ldots , \pi_{d-1}$ are of degree $2$ and the maps $f, h_0, h_1,\ldots ,h_{d-1}$ of degree $d$. Moreover, since $d$ is odd, all the involution in the dihedral group are conjugate to each other, therefore all the curves $C_i$ are isomorphic. Moreover $g(C_0)=g(C_1)=\ldots = g(C_{d-1})=k$.
In fact, all the maps $\pi_i$ ramify in six points, one in each the preimage by $f$ of the Weierstrass points of $C$. In particular, $\pi_i^*:JC_i \ra J\widetilde C$ is injective. From now on, we identify $JC_i$ with its image in $J\widetilde C$.

\begin{prop}\label{Propr Ries} (Ortega, Ries). With the notations in the diagram (\ref{global_diagram}) and denoting by $P$ the Prym variety $P(\widetilde C,C)$, the following statements hold for any $i=0,\ldots , d-1$:
\begin{enumerate}
    \item [a)] $JC_i\subset P$.
    \item [b)] The automorphism $\sigma^i$ sends $JC_0$ to $JC_{d-2i}$.
    \item [c)] The automorphism $\beta_i:=\sigma^i+\sigma ^{-i}$ leaves invariant $JC_0$.
\end{enumerate}
 \end{prop}

A crucial ingredient for the proof of our main theorem is the existence of $k$ isomorphisms:
\[
\psi_i: JC_0 \times JC_0 \lra JC_0 \times JC_i \lra P, 
\]
where the first map sends $(x,y)$ to $(x,\sigma ^i(y))$ (for $i=1,\ldots ,k$) and the second is the addition map. 

Let us denote by $\lambda_N:JN\ra JN^{\vee}$ the isomorphism attached to the natural principal polarization on a smooth irreducible curve $N$. We will keep this notation for the rest of the paper.

The pull-back of the polarization $\tau $ to $JC_0\times JC_0$ gives rise to the following commutative diagram:
\begin{equation}\label{pullback_tau}
\xymatrix@C=2cm@R=1.4cm{
    &JC_0 \times JC_0 \ar[ld]_{M_i} \ar[r]^{\psi_i} \ar[d]^{\lambda_{\psi_i^*(\tau)}}   & P \ar[d]^{\lambda_\tau} \\
JC_0\times JC_0 
\ar[r]^{\lambda_{C_0}\times  \lambda_{C_0}}
& JC_0^{\vee}\times JC_0^{\vee }&P^{\vee} \ar[l]_{\psi_i^{\vee }}
	      }
\end{equation}
where $M_i$ is the matrix
\[
\left(
\begin{array}{cc}
 2     & \beta_i \\
  \beta_i   & 2
\end{array}
\right),
\]
and the automorphisms $\beta_i$, for $i=1,\ldots,k$ of $JC_0$ are those appearing in the previous proposition.

It will also be useful to know the pull-back of the automorphisms $\sigma ^i$ to $JC_0\times JC_0$ through the isomorphisms $\psi_i$. Indeed, we have the following:

\begin{prop} \label{pullback_aut} For any $i=1,\ldots,k$ we have the equality: 
\[
    \sigma^i \circ \psi_i= \psi_i\circ \begin{pmatrix}0 & -1\\1&\beta_i\end{pmatrix}.
\] 
\end{prop}

\begin{rem}
Notice that a priori $\psi_i^*(\langle \sigma\rangle)$ yields $d-1$ automorphisms of $JC_0\times JC_0$. The crucial point is that only one among them is of type \begin{equation*}
    \begin{pmatrix}
    0 & *\\1 & *
    \end{pmatrix}.
\end{equation*}
\end{rem}

\section{Proof of the main Theorem}\label{section 3}

Our aim is to prove a generic injectivity of the map 
\[
\mathcal P_d:\mathcal R_d \lra \mathcal A_{d-1}^{(1,\ldots,1,d)},
\]
with $d=2k+1$, assuming $k$ and $d$ primes.

\subsection{First step: uniqueness of $\sigma $.}

We want to use the automorphism  $\sigma $  in order to reap from $(P,\tau)$ relevant information to reconstruct $(\widetilde C,C)$.  
We shall show that the only automorphisms of $P$ of order $d$ preserving $\tau$ and in fact,  fixing  point-wise $K(\tau)$ (see Lemma \ref{fixed_points}) are 
the powers 
$\sigma^i $. 

\begin{prop}\label{autom Prym}
   Let $P$ be a general element in $\Image(\mathcal{P}_d)$. Then the group of automorphisms 
   $$\{ \epsilon \in \Aut((P,\tau)) \ \mid \ \epsilon (x)=x, \ \forall x\in \Ker(\lambda_{\tau})   \}$$
 is isomorphic to $\langle \sigma \rangle \simeq \mathbb Z/d\mathbb Z$. 
    \begin{proof}
        Let $P$ and $\epsilon \neq \mathrm{Id}$ be as in the statement. We will see that $\epsilon =\sigma ^i$ for some $i=1,\ldots ,d-1$, where $\sigma $ is as in section 2. Due to Lemma \ref{fixed_points}, there is an automorphism $\tilde \epsilon: J\widetilde C \ra J\widetilde C$ such that the following diagram is commutative:
          \begin{equation*}
     \begin{tikzcd}
 	0\arrow{r} &K(\tau) \arrow{r} \ar[equal]{d}  & f^*JC\times P \arrow{r}{\mu} \arrow{d}{(\Id, \epsilon)} &J\widetilde{C}  \arrow{r}\arrow{d}{\tilde \epsilon} &0\\
  0 \arrow{r} &K(\tau) \arrow{r} & f^*JC\times P \arrow{r}{\mu} & J\widetilde{C}  \arrow{r} &0,
 	\end{tikzcd}
 \end{equation*}
 where $\mu:f^*JC \times P \ra J\widetilde C$ stands for the addition map.
  According to the diagram, $\mu^* \tilde \epsilon ^* \mathcal O_{J\widetilde C}(\widetilde \Theta)$ equals, as polarizations, $\mu^*\mathcal O_{J\widetilde C}(\widetilde \Theta )$. Since $\mu $ is an isogeny, the Kernel of $\mu^*$ is finite and therefore 
  \[
 \tilde \epsilon ^{\,*} \mathcal O_{J\widetilde C}(\widetilde \Theta)\otimes 
  \mathcal O_{J\widetilde C}(-\widetilde \Theta ),
  \]
  is a torsion sheaf, in particular belongs to $\Pic^0(J\widetilde C)$. Hence
  $\tilde \epsilon ^* \mathcal O_{J\widetilde C}(\widetilde \Theta)$ induces the canonical polarization on $J\widetilde C$ and thus, by Torelli Theorem, there is an automorphism $\tilde \epsilon_0$ on $\widetilde C$ inducing $\tilde \epsilon$. Notice that, by construction,  $\tilde \epsilon $ is the identity on $f^*(JC)$, therefore $\tilde \epsilon_0$ belongs to the Galois group of $\widetilde C$ over $C$ which is the cyclic group generated by $\sigma$.
    \end{proof}
\end{prop}

\subsection{Second step: determination of the automorphisms on $C_0$}
 This subsection is devoted to  the study of the possible automorphisms on $C_0$ that will appear in the third step. 
 From now on, the covering $\widetilde C\ra C$ is generic in $\mathcal R_d$.
 Given an abelian variety $A$, we denote by  $\End(A)$ its endomorphism ring and by  $\End_0(A)$ the endomorphism algebra $\End(A)\otimes_{\mathbb{Z}}\mathbb{Q}$. 
 \begin{defn}\label{var GL2 type}
  The abelian variety $A$ is said to be of $GL_2$-type if for some number field $E$ such that $[E:\mathbb{Q}]=\dim A$, there is an embedding of  $\mathbb{Q}$-algebras $E\hookrightarrow\End_0(A)$.
 \end{defn}
 \begin{defn}\label{var CM type}
  The abelian variety $A$ is said to be of CM-type if for some CM field $E$ such that $[E:\mathbb{Q}]=2\dim A$, there is an embedding of  $\mathbb{Q}$-algebras $E\hookrightarrow\End_0(A)$.
 \end{defn}
 \begin{prop}\label{decomposition JC_0}
     The Jacobian variety $JC_0$ is of $GL_2$-type. 
     \begin{proof}
         This is a straightforward check of Definition \ref{var GL2 type} for the Jacobian variety $JC_0$. The automorphism $\beta_1$ determines the subfield $E:=\mathbb{Q}(\xi+\xi^{-1})\subseteq\End_0{JC_0}$, where $\xi$ is a primitive $d$-root of the unity. Since $[E:\QQ] =(d-1)/2=k=\dim JC_0$ the claim follows. 
     \end{proof}
 \end{prop}
\begin{rem}
Notice that the subfield $E$ is totally real and that the Rosati involution acts here as the identity.
\end{rem}
Hence, we can state the following:
\begin{prop}\label{order_2}
    Any automorphism of $C_0$ has order 2.
 \end{prop}
  \begin{proof}
 Let $\phi$ be an automorphism of $C_0$. Let us assume that it has order prime $p \geq 3$. This yields the inclusion $\mathbb{Q}(\zeta_p)\hookrightarrow\End_0(JC_0)$, where $\zeta_p$ is a $p$-th primitive root of the unity. It is well known that for every $m| (p-1)=|Gal(\mathbb{Q}(\zeta_p)/\mathbb{Q})|$ there exists a subfield $K_m\subseteq \mathbb{Q}(\zeta_p)$ with degree $m$ over $\mathbb{Q}.$ Thus, in particular, there exists a subfield $K_2:=\mathbb{Q}(\alpha)\subseteq \mathbb{Q}(\zeta_p)$ with degree $2$ over $\mathbb{Q}.$ So we have the inclusion $\mathbb{Q}(\alpha, \xi+\xi^{-1})\subseteq\End_0{JC_0}$ with \[[\mathbb{Q}(\alpha, \xi+\xi^{-1}): \mathbb{Q}]= [\mathbb{Q}(\alpha, \xi+\xi^{-1}):\mathbb{Q}(\xi+\xi^{-1})][\mathbb{Q}(\xi+\xi^{-1}): \mathbb{Q}] = 2k.\]
 Therefore, $JC_0$ is a CM-type abelian variety, but this is impossible since there are only countably many such abelian varieties. 
 
 By analogous reasons, we can exclude the case $\text{ord}(\phi)=2^t$ with $t>1$. If  $\text{ord}(\phi)$ is not prime, we can factorize $\phi$ through maps of smaller prime order and thus conclude the result. Therefore, it only remains the case $p=2$.    
  \end{proof}

 In \cite{wu}, the author shows the following:
 \begin{prop}[\cite{wu}, Proposition 1.5]\label{prop wu}
     Let $A$ be an abelian variety of $GL_2$-type. 
     \begin{itemize}
         \item[1] If $A$ is not a CM abelian variety, then $A$ is isogenous to $A_1^r$, where $A_1$ is a simple abelian variety of $GL_2$-type and $r\in\mathbb{N}$.
         \item[2] If $A$ is a CM abelian variety, then $A$ is isogenous either to $A_1^r$, where $A_1$ is a simple CM abelian variety and $r\in\mathbb{N}$,  or to $A_1^{r_1}\times A_2^{r_2}$, where $A_i$ is a simple CM Abelian variety and $r_i\in\mathbb{N}$ for $i = 1, 2$ and $r_1 \dim A_1 = r_2 \dim A_2$.
         \end{itemize}
 \end{prop}
Applying this result to our situation we
have the following: \begin{prop}\label{automorphism}
     Assume $k$ be a prime number. Then, any automorphism $\phi$ of $C_0$ is either the identity  or, potentially, the hyperelliptic involution. In particular, the induced automorphism on $JC_0$ is $\pm Id$.
     \begin{proof}
     Due to Proposition \ref{order_2}, we can assume that the automorphism is an involution.
         Since there are only countably many CM abelian varieties, by Proposition \ref{prop wu}, we have the following two possibilities: either $JC_0$ is simple or it is isogenous to $A_1^r$.  Suppose that $JC_0$ is simple: either $\phi=\Id$,  or the quotient of the covering map $C_0\ra C_0/\langle \phi \rangle$ is $\mathbb{P}^1$. In the latter case $C_0$ is hyperelliptic and thus, on $JC_0$, the automorphism is $-\Id$. Suppose now that $JC_0$ is not simple, namely  $JC_0\sim A_1^r$. This yields $k=r\dim A_1$. Since $k$ is prime, the only possibility is $r=k$ and $\dim A_1=1$. This leads to a contradiction since $C_0$ varies in a 3-dimensional family whereas the moduli space of elliptic curves is  1-dimensional.  
     \end{proof}
 \end{prop}

\subsection{Third step: recovering $(C_0,\beta_1,\ldots, \beta_k)$}

 The dihedral construction of diagram \eqref{global_diagram} 
 gives a morphism $$\psi: \mathcal{R}_d \ra \mathcal{M}_{k}$$
 sending $[C,\eta]$ to $[C_0]$. Notice that this map is well-defined, since all the curves $C_0,\ldots ,C_{d-1}$ are isomorphic. Observe that the curve $C_0$ determines $P$ as abelian variety, but it is not possible to construct the  polarization $\tau $ on $P$ from the principal polarization on $JC_0$. Instead, according to the diagram (\ref{pullback_tau}), the automorphism $\beta_i$ are enough to define $\tau$. Let us consider the moduli space, $\widetilde {\mathcal D}_k$, of the isomorphism classes of objects $(C_0,\beta_1,\ldots ,\beta_k)$ and let  $\widetilde {\mathcal D}_k\ra \mathcal D_k$ be the forgetful map $(C_0,\beta_1,\ldots ,\beta_k)\mapsto C_0$.  Since the pair $(P,\tau)$ can be constructed from this data (see diagram (\ref{pullback_tau})), we have a factorization of the Prym map $\mathcal P_d$ as follows:
  \begin{equation*}
     \begin{tikzcd}
 	\mathcal R_d \arrow{rrr}{\mathcal P_d}\arrow[ddr, bend right, "\mathcal P_{d,1}"] \arrow{dr}{\psi} &   &   & \mathcal A_{d-1}^{(1,\ldots,1,d)}  \\
            &  \qquad \quad \mathcal D_k \subset \mathcal M_k&  & \\
       & \widetilde {\mathcal D}_k \arrow [u]  \arrow[uurr, bend right, "\mathcal P_{d,2}"] & &
 	\end{tikzcd}
 \end{equation*}
 
 \begin{rem}
     Using this factorization, Albano and Pirola showed that the generic fibers of $\mathcal P_3$, and $ \mathcal P_5$ are positive dimensional (see \cite[Remark 2.8]{albano pirola}). 
 \end{rem} 
  
The aim of this step is to prove that $\mathcal P_{d,2} $ is generically injective. In the fourth step, we will show that $\mathcal P_{d,1}$ has also degree $1$.

We consider isomorphisms $\varphi: JN\times JN\ra P$, where $N$ is a smooth curve of genus $k$. 
We say that such an isomorphism $\varphi $ satisfies the property (*) if and only if the pull-back of $\tau $ is as in diagram (\ref{pullback_tau}), that is:
\begin{equation}
 \varphi ^{\vee }\circ \lambda_{\tau } \circ \varphi = \begin{pmatrix}
     2 \lambda_N &   \lambda_N \circ \gamma \\ \lambda_N \circ \gamma & 2 \lambda _N 
\end{pmatrix}, \tag{$*$} 
\end{equation} 
for some $\gamma \in \Aut(JN)$. In the same way, property (**) holds if $\varphi $ behaves as in Proposition \ref{pullback_aut}, i.e. for some automorphism $\gamma $ of $JN$ and some exponent $i$ we have:
\[
  \sigma ^i \circ \varphi =\varphi \circ \begin{pmatrix}
     0 & -1 \\ 1 & \gamma 
\end{pmatrix}. \tag{$**$}
\]

Then, we define the intrinsic set attached to $(P,\tau, \langle \sigma \rangle)$: 
\begin{equation}\label{set automorphisms}
   \Lambda (P,\tau, \langle \sigma \rangle):= \{ (N,\varphi) \mid \varphi: JN\times JN \stackrel{\cong}{\lra } P  \text{ satisfies } (*), (**) \text{ for the same } \gamma \in \Aut(JN)\}.  
   \end{equation}

\begin{prop}\label{recovering C_0 and beta}
    Let $(P,\tau, \langle \sigma \rangle)$ be generic. Then for all $\varphi \in \Lambda  (P,\tau, \langle \sigma \rangle)$, we have that:
    \begin{enumerate}
        \item [a)]  $N\cong C_0$;
        \item [b)]  $\gamma=\beta_i$ for an $i$ in $1,\cdots, k$.
    \end{enumerate}
    \begin{proof}
        Let $\varphi \in \Lambda (P,\tau, \langle \sigma \rangle)$. We fix $i$, the exponent appearing in $(**)$. The composition: 
 $$
\xymatrix@C=2cm@R=1.4cm{
F:JN\times JN \ar[r]^{\quad\varphi }&P \ar[r]^{\psi_i^{-1}\quad} & JC_0\times JC_0,
	 }
$$
whose associated matrix is of type $F=\begin{pmatrix}
     A & B\\ C & D
\end{pmatrix}$, provides an isomorphism such that the polarization in $JC_0 \times JC_0$ with matrix:
\[
\Lambda_0:=\begin{pmatrix}
    2 \lambda_{C_0} &  \lambda_{C_0} \circ \beta_i \\
    \lambda_{C_0} \circ \beta_i & 2 \lambda_{C_0} 
\end{pmatrix},
\]
pulls-back to the polarization on $JN \times JN$ with matrix:
\[
\Lambda_N:=\begin{pmatrix}
    2 \lambda_{N} &   \lambda_{N} \circ \gamma\\
     \lambda_{N} \circ \gamma  & 2 \lambda_{N} 
\end{pmatrix}.
\]
Considering the restriction to $JN\times \{0\}$ and projecting to the first factor $JC_0$, we get a diagram as follows: 

\begin{equation}\label{diagram isom JC_0}
    \begin{tikzcd}
 	JN\arrow[r, "\iota"]\arrow[d, "2\lambda_N"] &JN\times JN\arrow{d}
    {\Lambda_N}
     \ar[r, "F"] & JC_0\times JC_0\arrow[r,"pr_1"]\arrow{d}
    {\Lambda_0}
    & JC_0\arrow{d}{2\lambda_0}\\
   JN^\vee & JN^\vee\times JN^\vee \arrow[l,"pr_1"]    &  JC_0^\vee\times JC_0^\vee\arrow[l,"F^\vee"] & JC_0^\vee \arrow [l,"\iota"].
   \end{tikzcd}
 \end{equation}
Indeed, the dual of $\iota$ is $pr_1$ and the polarization induced by $\Lambda_N$ on $JN$ is exactly $\lambda_N$. Therefore, the map $pr_1\circ F\circ \iota: JN\ra JC_0$ satisfies that the pull-back of twice the canonical polarization on $JC_0$ is twice the canonical polarization on $JN$. Hence $(JN,\Theta_N)\cong (JC_0,\Theta _{C_0})$. By Torelli Theorem, $N\cong C_0$. This proves a). 

In order to prove b), we fix an isomorphism $N\cong C_0$ and we look again at diagram \eqref{diagram isom JC_0} (by an abuse of notation we still use the letter $F$). The composition $pr_1\circ F\circ \iota: JC_0\ra JC_0$ corresponds to  the piece $A$ in the matrix of $F$. Hence, \[2\lambda_0= 2A^\vee\lambda_0 A, \]
that is, $A$ preserves the polarization on $JC_0$. Thus, $A$ comes from an automorphism of $C_0$. By Proposition \ref{automorphism}, we obtain $A=\pm \Id$. Replacing $\iota$ by $\iota_2$ (the restriction to the second factor of $JC_0\times JC_0$) and, analogously, $pr_1$ by $pr_2$ in diagram \eqref{diagram isom JC_0}, we obtain $D=\pm \Id$. In order to study the terms $B$ and $C$ recall that, by assumptions (property (**)), for a certain unique $i$, we have  \begin{equation}\label{autom on product}
    \sigma^i\circ \varphi=\varphi\circ \begin{pmatrix}
    0 & -1\\
    1 & \gamma
    \end{pmatrix},
\end{equation}
    and, thanks to Proposition \ref{pullback_aut}, the same occurs for $\psi_i$ with $\sigma_{\beta_i}:=\begin{pmatrix}
    0 & -1\\
    1 & \beta_i
    \end{pmatrix}$. Therefore, we have the following diagram: 
        \begin{equation}\label{diagram automorf}
    \begin{tikzcd}
    JC_0\times JC_0 \arrow[r, "F"]\arrow[d, "\sigma_{\gamma}"] & JC_0\times JC_0 \arrow[d, "\sigma_{\beta_i}"] \\
    JC_0\times JC_0 \arrow[r, "F"] & JC_0\times JC_0,
    \end{tikzcd}
    \end{equation}
    where $\sigma_{\gamma}$ is the matrix in \eqref{autom on product}. Now we analyze the possibilities for $F$. 
    \begin{enumerate}
        \item Let $F=\begin{pmatrix}
    \Id & B\\
    C & \Id
    \end{pmatrix}$. The commutativity of the diagram above says that $B=-C=0$ and $\gamma=\beta_i$. 
    \item Let $F=\begin{pmatrix}
    \Id & B\\
    C & -\Id
    \end{pmatrix}$, in this case $B=-C$ and $\gamma=\beta_i$.
        \end{enumerate}
        And the same occurs for the other configurations. This ends the proof. 
    \end{proof}
\end{prop}

\subsection{Fourth step: theta duality.}
The first three steps show  that the curve $C_0$ and the set of automorphisms $\{\beta_1, \ldots ,\beta_k\}$ of $JC_0$ can be recovered from the initial data. We shall prove that these automorphisms determine the map $h_0:C_0 \ra \mathbb P^1$ appearing in the diagram \eqref{global_diagram}. 


Let us recall the definition of the theta-dual of a subvariety $X$ of a principally abelian variety $(A,\tau_A)$, with $\dim(X) \leq \dim A -2$. Fix an effective theta divisor $\Theta $ representing the polarization $\tau_{A}$. 
\begin{defn} The {\it theta-dual} $T(X)$ of $X$ is set-theoretically defined by
\[
T(X)=\{a\in A \mid X\subset \Theta +a\}.
\]
\end{defn}
Pareschi and Popa gave a natural scheme structure to $T(X)$ for any closed reduced subscheme $X$ (see \cite[Def. 4.2]{pp})
by means of  the Fourier-Mukai transform on $A$.  They also proved (see \cite[Section 8]{pp}) that for a smooth curve $N$ embedded in its Jacobian it holds, up to traslation,  $T(N)=-W_{g(N)-2}(N)$, where $W_{g(N)-2}(N)$ stands for the Brill-Noether locus of effective divisors of degree $g(N)-2$. Similar results for Prym curves in their Prym varieties are obtained in \cite{ln}.

The key idea in this part is to compute the theta-dual of the curve $\beta_i(C_0)$ embedded in $JC_0$ and use this to recover  the whole map $h_0$.   

Let us fix any point $x \in C_0$ and consider the canonical injection $\iota_x: C_0\hookrightarrow JC_0, p\mapsto [p-x]$. We consider the canonical representative of the theta divisor in $\Pic^{k-1}(C_0)$ as $W_{k-1}(C_0)$, recall that $C_0$ has genus $k$. Then we can represent $T(\iota_x(C_0))$ as the set of $\xi \in  \Pic^{k-1}(C_0)$ such that $\xi + \iota_x(C_0)\subset W_{k-1}(C_0)$.
With these notations, we have:
 \begin{prop} \label{theta-dual}
 Assume that $k \ge 4$, then the following equality holds in $\Pic^k(C_0)$:
     \[
      T(\beta_i(\iota_x(C_0)))=x_i+x_{d-i}+W_{k-3}(C_0).
     \]
 \end{prop}
 \begin{proof}
     Observe that to use the definition of $\beta_i$ we need to see $JC_0=\pi_0^*(JC_0)$ as a subvariety of $P$. Given a point $p$ in $C_{0}$, $\iota_x(p)=[p-x]$ appears as $[p'+j(p')-x'-j(x')]\in P$, where $x',p'\in \widetilde C$ are preimages of $x,p$, respectively. 
     We denote by 
     \[
     p', p'_1:=\sigma (p'),\ldots ,p'_{d-1}=p'_{2k}:=\sigma ^{d-1}(p'),
     \]
the whole fiber $f^{-1}(f(p'))$, and analogously for $x'$. 
     Then, the action of $\beta_i=\sigma^i+\sigma^{-i}$ on $[p'+j(p')-x'-j(x')]$ is as follows:
           \begin{align*}
     &     [\sigma ^i(p')+\sigma^i(j(p'))+\sigma ^{-i}(p')+\sigma^{-i}(j(p'))
     -\sigma ^i(x')-\sigma^i(j(x'))-\sigma ^{-i}(x')-\sigma^{-i}(j(x'))]= \\
     &  [p'_i+j(p'_{d-i})+p'_{d-i}+ j(p'_{i}) - x'_i-j(x'_{d-i})-x'_{d-i}-j(x'_{i})]= \\ 
     & (1+j)([p'_i+p'_{d-i}-x'_i-x'_{d-i}]).
     \end{align*}
     Then, as element in $JC_0$, we have obtained that:
     \[
        \beta_i ( [p-x])=[p_i+p_{d-i}-x_i-x_{d-i}],
     \]
     where $x_i$ (resp. $p_i$) is, by definition, $\pi_0(x'_i)$ (resp.  $\pi_0(p'_i)$). Observe that these points describe  the fiber $h_0^{-1}(h_0(x))$, more precisely, as divisors:
      \begin{equation}\label{fibre_h_0}
      h_0^{-1}(h_0(x))=x+x_1+\ldots +x_{d-1}.
      \end{equation}
      By definition, $\xi \in T(\beta_i(\iota_x(C_0)))  $ means that:
     \[
     h^0(C_0,\xi + p_i+p_{d-i}-x_i-x_{d-i})>0,\qquad \text{for all } p\in C_0.
     \]
     If $\xi $ is of the form $x_i+x_{d-i}+E$ for some effective divisor $E$ the condition is satisfied trivially, hence $x_i+x_{d-i}+W_{k-3}(C_0)\subset  T(\beta_i(\iota_x(C_0))).$ 
 To prove the opposite inclusion, we consider $\xi\in T(\beta_i(\iota_x(C_0)))$.  If $h^0(C_0, \xi-x_i-x_{d-i})>0$ then we are done. So assume $h^0(C_0, \xi-x_i-x_{d-i})=0$. Set $L:= K_{C_0}-(\xi-x_i-x_{i-d})$. By Riemann-Roch Theorem, 
 $$h^0(C_0, L)=2k-2-(k-3)-k+1=2.$$ Hence, $L$ is a line bundle of degree $k+1$ with $h^0(C_0,L)=2$ and such that $h^0(C_0, L-p_{i}-p_{d-i})>0$ for all $p_i, p_{d-i}\in C_0$, namely, $L$ gives a $g^1_{k+1}$ with $p_i, p_{d-i}$ in the same fiber. This yields a contradiction. Indeed, changing the point $p$ with all other  points in $h_0^{-1}(h_0(p)))$ we obtain that the whole fibers of $h_0$ are contained in the $g^1_{k+1}$, which is impossible since $h_0$ has degree $d=2k+1$.
\end{proof}

\begin{proof}[Proof of Theorem \ref{main-thm}]
According to Proposition \ref{theta-dual}, for a fixed $x\in C_0$ we obtain intrinsically the subset $x_i+x_{d-i}+W_{k-3}(C_0)$. A  classical result of Weil (cf. \cite[Hilfssatz 3]{we}) states that $W_{k-3}(C_0)$ is not "translation invariant", that is, if $\alpha \in JC_0$ satisfies that $\alpha + W_{k-3}(C_0)=W_{k-3}(C_0)$, then $\alpha =0$. 
This says that   $x_i+x_{d-i}$ is uniquely determined in $C_0$. Varying $i=1,\ldots ,k$ and using 
(\ref{fibre_h_0}) we get that for all $x\in C_0$ the fiber  of $h_0$ containing $x$ is recovered. Now, by construction the map $h_0$ is ramified over six points in $\PP^1$, which are also the images of the Weierstrass points of $C$ by $\varepsilon $. Therefore, $\varepsilon : C\ra \mathbb P^1$ is determined by $h_0$. The fiber product of $\varepsilon: C\ra \mathbb P^1$ and $h_0:C_0 \ra \mathbb P^1$ gives
the map $\widetilde C \ra C$. This finishes the proof.
\end{proof}

	\section{Case $d=9$}\label{d=9}
	This section is devoted to the analysis of the case $d=9$, 
	which is no longer a prime number. Due to the structure of the dihedral group $D_9$, the corresponding diagram  is more complicated since new intermediate quotients of the curve $\widetilde C$ and of the curves $C_i$ appear. Indeed, the situation can be summarized in the following diagram:

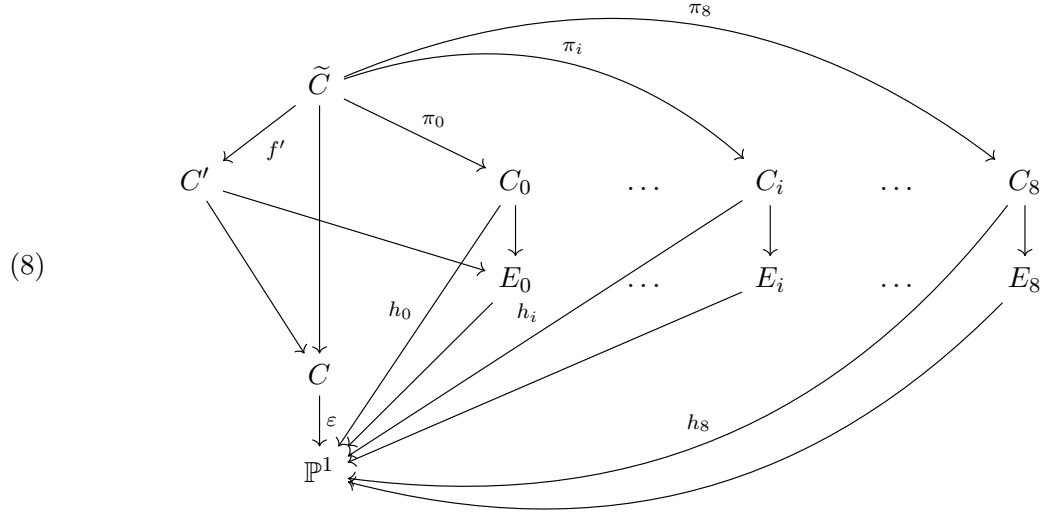
\begin{equation}
    \label{diagram case 9}
 \begin{tikzcd}
 	&\widetilde C \arrow[dl,"f'"]\arrow[ddd] \arrow[drr,"\pi_0"] 
  \arrow[drrrr,bend left,"\pi_{i}"] \arrow[drrrrrr,bend left,"\pi_{8}"]
  &&&&&& \\
 	C'  \arrow [ddr]\arrow [drrr] & &  & C_0 
  \arrow[dddll,"h_0"']
  \arrow[d]
  & \ldots   & C_i \arrow[dddllll,"h_{i}"']\arrow[d] & \ldots & C_{8} \arrow[dddllllll,bend left,"h_{8}"']\arrow[d]  
    \\
  &&& E_0\arrow[ddll] &\ldots & E_i\arrow[ddllll] &\ldots & E_8\arrow[ddllllll,bend left] \\
    & C \arrow[d,"\varepsilon"] &&&&&&\\
  & \mathbb P^1 &&&&&&
 \end{tikzcd}
 \end{equation}

The curves $C_i$ correspond to the quotients $\widetilde C/ \langle j \sigma^i\rangle$, the curve $C':=\widetilde C/\langle \sigma^3\rangle $ and, finally, the curves $E_i$ are obtained as $\widetilde C/\langle j\sigma^i, \sigma^3\rangle$. Hence $E_i$ and $ E_j$ are exactly the same curve if $i-j$ is a multiple of $3$.  The map $f'$ is étale of degree 3, while the maps $\widetilde C \ra E_i$ (composing $\widetilde C \ra C_i$ with $C_i\ra E_i$) are Galois of degree 6 with Galois group $\langle j\sigma^i, \sigma^3\rangle\cong D_3$. We have that $g(C')=4$, $g(C_i)=4$ and $g(E_i)=1$.  By congruence modulo 3, it is sufficient to consider the first three quotient curves $C_0, C_1,C_2$ (respectively $E_0, E_1, E_2$). Moreover, since the three dihedral groups $\langle j, \sigma^3\rangle, \langle j\sigma, \sigma^3\rangle$, and $\langle j\sigma^2, \sigma^3\rangle$ are conjugated in $D_9$, all the curves $E_i$ are isomorphic. Finally, notice that the non-Galois degree 9 maps $h_i: C_i\ra \PP^1$ factor as non-cyclic triple covers $f_i:C_i\ra E_i$ plus a degree 3 map $E_i\ra \PP^1$ ramified exactly over the Weiestrass points of $C$. The map $f_0: C_0\ra E_0$ decomposes up to isogeny $JC_0$ as the product $ E_0\times P(C_0,E_0)$, where $P(C_0,E_0)$ is the 3-dimensional abelian variety defined as the Prym variety associated with $f_0$. 

 A step by step analysis of Proposition \ref{Propr Ries} and Proposition \ref{pullback_aut} shows that they remain true in case $d=9$ too. Indeed, we have: 
\begin{prop}\label{prop caso 9}
    For $i=1,\dots, 4$ the following properties hold.
    \begin{enumerate}
        \item The automorphisms $\beta_i=\sigma^i+\sigma^{-i}$ restrict to automorphisms of $JC_0$.
        \item The maps $\psi_i: JC_0\times JC_0\ra P$, sending $(x,y)$ to $x+\sigma^i(y)$ are isomorphisms of polarized abelian varieties such that \begin{itemize}
            \item 
              $\psi_i ^{\vee }\circ \lambda_{\tau } \circ \psi_i = 
              \begin{pmatrix} 2 \lambda_{C_0} &   \lambda_{C_0} \circ \beta_i \\  \lambda_{C_0} \circ \beta_i & 2 \lambda _{C_0} 
              \end{pmatrix}$;
            \item 
              $\psi_i ^{-1}\circ \sigma^i \circ \psi_i = \begin{pmatrix}
                0 &-1\\ 1 &\beta_i 
            \end{pmatrix}$.    
        \end{itemize}
    \end{enumerate}
\end{prop}
Since both $\pi_0$ and $f_0$ are ramified maps, we have the inclusions $E_0\subset JC_0\subset J\widetilde C$. More precisely, we have the following:

\begin{prop}\label{E_0 and map f_0} The map $f_0: C_0\ra E_0$ is given by the composition $C_0\hookrightarrow JC_0\xrightarrow{(1+\beta_3)} E_0$. In particular, $E_0=\Image(1+j)(1+\beta_3).$
    \begin{proof}
        Let us focus on the following part of diagram  \eqref{diagram case 9}:
        \begin{equation}\label{diagram D_3}
     \begin{tikzcd}
     \widetilde C \arrow[d,"f'"]\arrow[r,"\pi_0"] \arrow[dr,"6:1"]& C_0= \widetilde C/\langle   j\rangle \arrow[d,"f_0"] \\
     C'= \widetilde C/\langle   \sigma^3\rangle \arrow[r] & E_0=\widetilde C/\langle   \sigma^3, j\rangle.
 \end{tikzcd}
 \end{equation}
Let \[
 x, \sigma(x), \dots, \sigma^8(x)\quad \text{resp.} \quad  jx, \sigma(jx), \dots, \sigma^8(jx)
\]  be the whole fibers $f^{-1}(f(x))$, resp. $f^{-1}(f(jx))$. By construction,  the map $f'$ identifies $[x]=[\sigma^3(x)]= [\sigma^6(x)]$. By the commutativity of the diagram, we also have $E_0= C'/\langle j\rangle$. Thus, in $E_0$, we get the $6:1$ identification \begin{equation}\label{6:1 map}
    [x]=[\sigma^3(x)]= [\sigma^6(x)]=[jx]=[j\sigma^3(x)]= [j\sigma^6(x)].
\end{equation}Obviously, the analogous holds true for the fiber of $\sigma(x) $ and of $\sigma^2(x)$. 

By definition, in $C_0$, we have $[x]=[jx]$, $[\sigma^3(x)]=[j\sigma^3(x)]=[\sigma^6(jx)]$, and $[\sigma^6(x)]=[j\sigma^6(x)]=[\sigma^3(jx)]$. The commutativity of \eqref{diagram D_3}, together with \eqref{6:1 map}, gives us the 3:1 map $f_0$. As we already know, $JC_0=\Image(1+j)$. Thus, we get $JE_0\cong E_0=\Image(1+j)(1+\sigma^3+\sigma^6)$. Finally, it is easy to check that composing with the natural injection $\iota_x: C_0\hookrightarrow JC_0, \, p\mapsto [p-x]$, we obtain the map $f_0$. 
    \end{proof}
\end{prop}
As a consequence, we get the following:
\begin{prop}\label{beta_i on elliptic curves}
    The automorphisms $\beta_1,\beta_2,\beta_4$ restrict to $-\Id$ on $E_0$, while $\beta_3$ acts as multiplication by 2.
    \begin{proof}
        It is a straightforward computation using that $E_0=\Image (1+j)(1+\sigma^3+\sigma^6)$, the fact that $\beta_i$ commutes with $(1+j)$ and that $1+\sigma+\dots+\sigma^8=0$ on $J\widetilde C$. 
    \end{proof}
\end{prop}
Let $\xi$ be a 9th primitive root of unity. In what follows, we want to describe more explicitly the automorphisms $\beta_i$. First, we recall that the following equalities hold:\begin{equation}\label{degree extension}
    [\mathbb{Q}(\xi): \mathbb{Q}]= [\mathbb{Q}(\xi): \mathbb{Q}(\xi+\xi^{-1})][\mathbb{Q}(\xi+\xi^{-1}): \mathbb{Q}]=2\cdot 3=6.
\end{equation}
The action of $\sigma$ on $H^0(\widetilde C, \omega_{\widetilde C})$ decomposes as follows (\cite[Section]{lo11}) 
\begin{equation}
    H^0(\widetilde C, \omega_{\widetilde C})= \bigoplus_{i=0}^8 H^0( C, \omega_{ C}\otimes \eta^i),
\end{equation}
where $H^0( C, \omega_{ C}\otimes \eta^i)$ is the eigenspace corresponding to the eigenvalue $\xi^i$. This yields the following:
\begin{prop}
    The automorphisms $\beta_i$ act on $T_0JC_0$ as the diagonal matrix $$diag(\xi^i+\xi^{-i}, \xi^{2i}+\xi^{-2i}, \xi^{3i}+\xi^{-3i}, \xi^{4i}+\xi^{-4i}).$$ 
\end{prop}
Therefore, up to permutation of the eigenspaces, the automorphisms $\beta_1,\beta_2,\beta_4$ correspond to the matrix ${diag}(\xi+\xi^{8}, \xi^2+\xi^{7}, \xi^4+\xi^5, -1)$ while $\beta_3$ is given by $diag(-1,-1,-1,2)$. As an immediate consequence, we get the following:
\begin{prop}\label{prym GL2}
    The Prym variety $P(C_0, E_0)$ is of $GL_2$-type.
    \begin{proof}
        Using the description of the automorphisms $\beta_i: JC_0\ra JC_0$ as a diagonal matrix, together with Proposition \ref{beta_i on elliptic curves}, we obtain that they restrict to automorphisms of $P(C_0, E_0)$. Therefore $F:=\mathbb{Q}(\xi+\xi^{-1})$ is embedded in $\End_0(P(C_0, E_0))$. Since, by \eqref{degree extension}, $F$ is a totally real cubic field, we conclude that $P(C_0, E_0)$ is of $GL_2$-type.  
    \end{proof}
\end{prop}
Now we have all the ingredients to prove the following:
\begin{thm}\label{thm_d=9}
    The Prym map $\mathcal{P}_9: \mathcal{R}_9\ra \mathcal{A}_8^{(1,\dots,1,9)}$ is generically injective.
    \begin{proof}
        The proof traces the one given in case of $d, k$ prime numbers. The main difference here is that instead of recovering  $h_0: C_0\ra \PP^1$, we recover the map $f_0: C_0\ra E_0$. 

        Let $(P,\tau)$ be a general element in $\Image(\mathcal{P}_9)$. By Proposition \ref{autom Prym}, we obtain the triplet $(P, \tau, \langle\sigma\rangle)$. Thus, as in \eqref{set automorphisms}, we define the set $\Lambda(P,\tau, \langle\sigma\rangle)$. Since, by Proposition \ref{prop caso 9} also in case of $d=9$
    the maps $\psi_i$ are isomorphisms of polarized abelian varieties, we can use Proposition \ref{recovering C_0 and beta} to recover the curve $C_0$. Using the same Proposition (the second part of its proof where property $(\ast\ast$) is exploited) one can  show that $\gamma$ is of the form  $\phi^{-1}\beta_i\phi$ for an $i \in \{ 1,2,3,4 \}$ and for $\phi$ an automorphism of $JC_0$ (compatible with the principal polarization -but this will not be used-). 

        Now, using the obtained automorphisms $\gamma$'s,  we consider the 1-dimensional images $\Image(\gamma+\Id)$. By Proposition \ref{beta_i on elliptic curves}, these images are the curves $E_0$ or $\phi(E_0)$. We do not take into account higher dimensional images, i.e. the ones coming from $\beta_i$ or from $\phi^{-1}\beta_i\phi$ for $i=1,2,4$. By Proposition \ref{E_0 and map f_0}, we get the map $f_0$ in case of $\gamma=\beta_3$, resp. the map $C_0\ra \phi(E_0)$ in case of $\gamma=\phi^{-1}\beta_3\phi$. In this way we conclude the result. Indeed, by definition, the Galois closure of $\widetilde C\ra E_0$ does not depend on the automorphism of the base. This yields the element $\widetilde  C\ra C \in \mathcal{R}_9$  we are looking for.
   \end{proof}
\end{thm}

\section{Case $d=13$} \label{d=13}
This section is devoted to the analysis of the case $d=13$, where $d$ is a prime number, whereas $k=6$ no longer is. This prevents us to use the full Proposition \ref{recovering C_0 and beta}, so one has to characterize the set $\Lambda(P,\tau, \langle\sigma\rangle)$ defined as in \eqref{set automorphisms}. The diagram to keep in mind is the following:
\begin{equation}\label{diagram case 13}
\begin{tikzcd}
&\widetilde C \arrow[dl,swap,"f"] \arrow[dr,"\pi_0"] &\\
C\arrow[dr,swap,"\varepsilon"] & & C_0 \arrow[dl,swap,"h_0"]\\
& \PP^1&
\end{tikzcd}
\end{equation}
We recall that $g(C_0)=6$ and that $\beta_i=\sigma^i+\sigma^{-i}$ are automorphisms of $JC_0$ not preserving the polarization. 
The second step of section \ref{section 3} yield the following:
\begin{prop}
    The Jacobian variety $JC_0$ is of $GL_2$-type which is either simple or is isogenous to $X^r$, where $X$ is a simple abelian variety of $GL_2$-type and $r\in \mathbb{N}$. Moreover, any automorphism of the curve $C_0$ is an involution.
    \begin{proof}
        It follows from Propositions \ref{decomposition JC_0}, \ref{order_2}, and \ref{prop wu}(\S 1). 
    \end{proof}
\end{prop}
We have the following:
\begin{cor}
    For a general $\widetilde C\ra C$, the Jacobian variety $JC_0$ is either simple or isogenous to the square of a simple abelian threefold. In the first case, any involution induced on $JC_0$ from an automorphism of $C_0$ is $\pm \Id$.
    \begin{proof}
    Indeed, we have to exclude that $JC_0$ is one of the following cases: 
        \begin{enumerate}
            \item $E^6$, where $E$ is an elliptic curve, 
            \item $S^3$, where $S$ is a $GL_2$-type abelian surface.
        \end{enumerate}
    This follows from a dimensional argument.  The first case is excluded because the moduli of elliptic curves is 1-dimensional. In the second case  the moduli of abelian surfaces is 3-dimensional, but the endomorphism algebra of the generic element is $\mathbb{Z}$ and this is not compatible with our assumptions on the endomorphism algebra of $S$, so it can be also ruled out. 
           \end{proof}
\end{cor}

Thus, in order to get the generic injectivity of $\mathcal P_{13}$, it only remains to exclude the case $JC_0\sim T^2$ (which prevent us to obtain a result as in Proposition \ref{recovering C_0 and beta}). Such configuration is realized when an involution $\phi$ acts on $C_0$ with quotient a genus 3 curve $C'_0$. That is, the 2:1 covering $C_0\ra C'_0=C_0/\langle\phi\rangle$, branched in $p_1+p_2$, describes an element $(C'_0, p_1+p_2, \eta)$ in $\mathcal{R}_{3,2}$ such that $T\sim JC'_0\sim P(C_0,C'_0)$ is a $GL_2$-type threefold. A negative answer to this question would allow to conclude the
generic injectivity of $\cP_{13}$.

\end{document}